\documentclass[12pt]{article}
\usepackage{amssymb}
\usepackage{mathrsfs}
\usepackage{amsfonts}
\usepackage{graphicx}
\usepackage{mathptmx}
\usepackage{colortbl}

\usepackage{latexsym,amsmath,amssymb,amsfonts,amsthm}

\newcommand{\R}{\mathbb{R}}
\newcommand{\h}{\mathbb{H}}
\newcommand{\s}{\mathbb{S}}

\newcommand{\boR}{\mathcal{R}}
\newcommand{\boL}{\mathcal{L}}

\newcommand{\Ome}{\Omega}

\newtheorem{theorem}{Theorem}[section]
\newtheorem{lemma}[theorem]{Lemma}

\newtheorem{remark}[theorem]{Remark}

\newtheorem{defn}[theorem]{Definition}

\newtheorem*{fact}{Fact}

\begin{document}
\setcounter{page}{1}
\title{$f$-extremal domains in hyperbolic space}
\author{Jos\'{e} M. Espinar$^{\dag}$, Alberto Farina$^{\ddag}$, Laurent Mazet$^{*}$}
\date{}
\maketitle ~~~\\[-15mm]

\begin{center}
{\footnotesize $^{\dag}$Instituto Nacional de Matem\'{a}tica Pura e Aplicada, 110 Estrada Dona Castorina, Rio de Janeiro, 22460-320, Brazil \\
Email: jespinar@impa.br\\}
{\footnotesize $^{\ddag}$LAMFA, CNRS UMR 7352, Facult\'{e} des Sciences, Universit\'{e} de Picardie Jules Verne  33, rue Saint Leu, 80039, Amiens CEDEX 1, France \\
Email: alberto.farina@u-picardie.fr\\}
{\footnotesize $^{*}$Laboratoire d'Analyse et Math\'{e}matiques Appliqu\'{e}es Universit\'{e} Paris-Est - Cr\'{e}teil, UFR des Sciences et Technologie, Avenue du G\'{e}n\'{e}ral de Gaulle, 94010, Cr\'{e}teil, France \\
Email: laurent.mazet@math.cnrs.fr\\}
\end{center}


\begin{abstract}
In this paper we study the geometry and the topology of unbounded domains in the Hyperbolic Space $\h ^n$ supporting a bounded positive solution to an overdetermined elliptic problem. 

Under suitable conditions on the elliptic problem and the behaviour of the bounded solution at infinity, we are able to show that symmetries of the boundary at infinity imply symmetries on the domain itself. 

In dimension two, we can strengthen our results proving that a connected domain $\Omega \subset \h ^2$ with $C^2$ boundary whose complement is connected and supports a bounded positive solution $u$ to an overdetermined problem, assuming natural conditions on the equation and the behaviour at infinity of the solution, must be either a geodesic ball or, a horodisk or, a half-space determined by a complete equidistant curve or, the complement of any of the above example. Moreover, in each case, the solution $u$ is invariant by the isometries fixing $\Omega$.

 \end{abstract}

\mbox{} 

\markright{\sl\hfill  J.M. Espinar, A. Farina, L. Mazet  \hfill}

\section{Introduction}
\renewcommand{\thesection}{\arabic{section}}
\renewcommand{\theequation}{\thesection.\arabic{equation}}
\setcounter{equation}{0} \setcounter{maintheorem}{0}

Solving an elliptic partial derivative equation under Dirichlet or Neumann data is a
classical problem but trying to impose both Dirichlet and Neumann data leads to a so
called overdetermined elliptic problem (OEP) and solutions should be very rare. For
example, consider the following problem in a domain (open and connected) $\Ome$ of $\R^n$,
\begin{eqnarray}\label{eq:oep}
\left\{
\begin{array}{llll}
\Delta{u}+f(u)=0  \quad&\text{in }   \Omega,\\
u>0  \quad&\text{in }   \Omega,\\
u=0  \quad&\text{on }\partial\Omega,\\
\langle\nabla{u},\vec{v}\rangle_{\R^n}=\alpha \quad&\text{on } \partial\Omega,
\end{array}
\right.
\end{eqnarray}
where $\vec\nu$ is the unit outward normal vector along $\partial\Ome$. A domain where the OEP~\eqref{eq:oep} can be solved is called a $f$-extremal domain. If $\Ome$ is bounded and $f\equiv 1$, Serrin~\cite{s} proved that the ball is the only domain where the above problem admits a solution $u$ (this was generalized later to any Lipschitz function $f$). The proof of Serrin uses the moving plane method that was introduced by Alexandrov in \cite{a} in order to prove that round spheres are the only constant mean curvature embedded hypersurfaces in $\R^n$.

In \cite{bcn}, Berestycki, Caffarelli and Nirenberg considered the above problem for
unbounded domains in $\R^n$ and proved that, under some additional hypotheses on $f$, the only $f$-extremal domain that is an epigraph is a halfspace. Moreover, they stated the following conjecture.

\medskip

\textbf{BCN conjecture.} If $f$ is Lipschitz, and $\Ome$ is a smooth connected domain with
$\R^n\setminus \overline\Ome$ connected where the OEP~\eqref{eq:oep} admits a bounded
solution, then $\Ome$ is either a ball, a halfspace, a cylinder $B^k\times\R^{n-k}$ ($B^k$
is a ball of $\R^k)$ or the complement of one of them.

\medskip

This conjecture is also motivated by the work of Reichel~\cite{Rei} concerning exterior
domains. Besides it has inspired many interesting results: for example, the works of
Farina and Valdinoci~\cite{fv1,fv2,fv3} about epigraphs or the one of Ros and
Sicbaldi~\cite{rs} concerning planar domains. Actually, in \cite{Sic}, Sicbaldi gave a
counterexample to BCN conjecture
in $\R^n$ for $n\ge 3$. But understanding the geometry of $f$-extremal domains is still an
interesting question and one of the main point is the similarity of the geometry of these
domains with the one of constant mean curvature hypersurfaces. Exploiting this similarity,
Ros, Ruiz and Sicbaldi \cite{RRS} proved that in dimension $2$ the BCN conjecture is true
for unbounded domains whose complement is unbounded: it has to be a halfplane.
We also refer to \cite{fv4} for the study of overdetermined elliptic problems on a complete, non-compact Riemannian manifold without boundary and with non-negative Ricci tensor.

In this paper, we are interested in the geometry of $f$-extremal unbounded domains in the
hyperbolic space. More precisely, let $\Omega \subset \h ^n $ be a domain (open and
connected) whose boundary, if not empty,
is of class $C^2$ and consider the following OEP
\begin{eqnarray} \label{1.4}
\left\{
\begin{array}{llll}
\Delta{u}+f(u)=0  \quad&\text{in }   \Omega,\\
u>0  \quad&\text{in }   \Omega,\\
u  \text{ bounded } & \text{in }  \Omega ,\\ 
u=0  \quad&\text{on }\partial\Omega,\\
u (p)\to C \quad & \text{uniformly as }  d( p ,\partial  \Omega ) \to +\infty\\
\langle\nabla{u},\vec{v}\rangle=\alpha \quad&\text{on } \partial\Omega,
\end{array}
\right.
\end{eqnarray}where $\langle\cdot,\cdot\rangle$ is the inner product on $\h ^n$ induced by
the hyperbolic metric, $d$ the hyperbolic distance, $\vec{v}$ the unit outward normal
vector along the boundary $\partial\Omega$ (we will also use the notation $\partial_\nu u$
for $\langle\nabla u,\vec\nu\rangle$), $\alpha$ a non-positive constant.

The function $f$ will be subject to the following assumptions:
\begin{align}
&f\textrm{ is Lipschitz,}\tag{H1}\label{hyp:h1}\\
&f\textrm{ is non-increasing}.\tag{H2}\label{hyp:h2}
\end{align}
Hypothesis~\eqref{hyp:h1} corresponds to the one in the BCN conjecture and is made all
along the paper so it would not be mentioned in the statements of the results. Concerning
Hypothesis~\eqref{hyp:h2}, it will not be assumed in some results of
Section~\ref{sec:dim2}, so we will precise when it is assumed.

The study of $f$-extremal domains in $\h^n$
 already appears in the work of Espinar and Mao~\cite{EM}. They
use the fact that the hyperbolic space can be compactified by its ideal boundary
$\partial_\infty\h^n$; so $f$-extremal domains can be studied in terms of their trace on
$\partial_\infty \h^n$, one of their results states that a $f$-extremal domain whose trace
on $\partial_\infty\h^n$ is at most one point is either a geodesic ball or a horoball
(hypothesis~\eqref{hyp:h2} is not need in this result), this generalizes, to any $f$,
results by Molzon~\cite{mr} and Sa Earp and Toubiana~\cite{SETou}.

Here, our study looks at $f$-extremal domains whose trace on $\partial_\infty\h^n$ is
larger, for example we prove that a
$f$-extremal domain whose complement is bounded is the complement of a geodesic ball. This result
is similar to the one of Reichel~\cite{Rei} in the Euclidean case. We also give a
characterization of the complement of a horoball.

We also prove that if the trace on $\partial_\infty\h^n$ of $\partial\Ome$ (with $\Ome$ is
a $f$-extremal domains) is some asymptotic equator (see Section~\ref{sec:hypgeo} for a
precise definition) then $\Ome$ has to be invariant by a big subgroup of hyperbolic
isometries. This result can be compared with the result of Berestycki, Caffarelli and
Nirenberg concerning epigraphs.

As mentioned above, the geometry of $f$-extremal domains seems to imitate the geometry of
constant mean curvature hypersurfaces. So it is interesting to compare our results with
the ones obtained by do~Carmo and Lawson in~\cite{dCL}  and Levitt and Rosenberg in \cite{LR} where constant mean curvature hypersurfaces in $\h^n$ are characterized by their trace on $\partial_\infty\h^n$.

The paper is organized as follows. In Section~\ref{sec:hypgeo}, we recall some aspects of
hyperbolic geometry and fix some notations used in the following sections.
Section~\ref{sec:mpm} is devoted to the study of exterior domains. In
Section~\ref{sec:inv}, we prove a result concerning the invariance by hyperbolic
translations of $f$-extremal domains. In the last section, we study $f$ extremal domains
$\Ome$ in $\h^2$ without hypothesis~\eqref{hyp:h2}, the main point is to understand the
asymptotic behaviour of $\partial\Ome$ when it is connected.


\section{Preliminaries about hyperbolic geometry}\label{sec:hypgeo}

In this section, we will give an exposition of some aspects of the Hyperbolic Space for
the reader convenience.


\subsection{The hyperbolic space and its ideal boundary}

The hyperbolic space $\h^n$ ($n\ge 2$) is (up to isometry) the only simply connected
manifold of constant sectional curvature $-1$. It is well known that the cut locus of any
point on $\h ^n$ is empty, which implies that for any two points on $\h ^n$ there is a
unique geodesic joining them. Therefore, the concept of geodesic convexity can be
naturally defined for subsets of $\h ^n $.

If $(p,v)$ is an element of the unit tangent bundle $U\h^n$, we define the half geodesic
starting from $p$ with initial speed $v$ as the geodesic $\gamma_v :[0,+\infty)\to \h^n$
with $\gamma_v(0)=p$ and $\gamma_v'(0)=v$. We say that two half geodesics $\gamma_1$ and
$\gamma_2$ are \emph{asymptotic} if there exists a constant $c$ such that the distance
$d(\gamma_1(t),\gamma_2(t))$ is less than $c$ for all $t\geqslant 0$. Similarly,
two unit vectors $v_1$ and $v_2$ are asymptotic if the corresponding geodesics
$\gamma_{v_1}(t)$, $\gamma_{v_2}(t)$ have this property. It is easy to find that being
asymptotic is an equivalence relation on the set of unit-speed half geodesics or on the set of
unit vectors on $\h ^n$. Each equivalence class is called a point at
infinity. Denote by $\partial _{\infty} \h ^n$ the set of points at infinity, and denote
by $\gamma(+\infty)$ or $v(\infty)$ the equivalence class of the corresponding geodesic
$\gamma(t)$ or unit vector $v$. It is called the end-point of $\gamma$. If $\gamma:\R\to
\h^n$ is a unit speed geodesic line, we
denote by $\gamma(-\infty)$ the equivalence class of $s\mapsto \gamma(-s)$.

It is well-known that for two asymptotic half geodesics $\gamma_1$ and $\gamma_2$ in $\h
^n$, the distances $d(\gamma_1(t),\gamma_2)$ and $d(\gamma_2(t),\gamma_1)$ goes to zero as
$t\to +\infty$. Besides, for any $x, y\in \partial_{\infty} \h ^n$, there exists a unique
oriented unit speed geodesic $\gamma$ such that $\gamma(+\infty)=x$ and
$\gamma(-\infty)=y$; this geodesic will be denoted by $(yx)$.

For any point $p\in \h ^n$, there exists a bijective correspondence between unit vectors
at $p$ and $\partial _{\infty} \h ^n$. In fact, for a point $p\in \h ^n$ and a point $x\in
\partial_{\infty} \h ^n$, there exists a unique oriented unit speed geodesic $\gamma$ such
that $\gamma(0)=p$ and $\gamma(+\infty)=x$. Equivalently, the unit vector $v$ at the point
$p$ is mapped to the point at infinity $v(\infty)$. Therefore, $\partial_{\infty} \h ^n$
is bijective to a unit sphere, i.e., $\partial _\infty \h ^n \equiv \s ^{n-1}$. For
$p\in\h^n$ and $x\in\partial_\infty\h^n$, we denote by $(px)$ the half geodesic starting
at $p$ with end-point $x$.

Set $\overline{\h ^n}=\h ^n\cup \partial _{\infty} \h^n$. For a point $p\in \h ^n$,
$\mathcal {U}$ an open subset of the unit sphere of the tangent space $T_{p}\h ^n$ and
$r>0$, define
 \begin{eqnarray*}
 T(\mathcal{U},r):=\{\gamma_{v}(t)\in \overline{\h ^n} \,
 |v\in\mathcal{U},~r<t\leqslant+\infty\}.
\end{eqnarray*}
Then there is a unique topology $\mathscr{T}$ on $\overline{\h ^n}$ with the following
properties: open subsets of $\h^n$ are open subsets of $\mathscr{T}$ and the sets
$T(\mathcal{U},r)$ containing a point $x\in \partial_{\infty} \h ^n$ form an open
neighborhood basis at $x$. This topology is called the \emph{ideal topology} of
$\overline{\h ^n}$. Clearly, the ideal topology $\mathscr{T}$ satisfies the following
properties:

(A1) $\mathscr{T}|_{\h ^n}$ coincides with the topology induced by the
Riemannian distance;

(A2) for any $p\in \h ^n$ and any homeomorphism $h:[0,1]\rightarrow[0,+\infty]$, the
function $\varphi$, from the closed unit ball of $T_{p}\h ^n$ to $\overline{\h ^n}$, given
by $\varphi(v)=\exp_{p}(h(\|v\|)v)$ is a homeomorphism. Moreover, $\varphi$ identifies
$\partial_{\infty} \h ^n$ with the unit sphere;

(A3) for a point $p\in \h ^n$, the mapping $v\rightarrow v(\infty)$ is
a homeomorphism from the unit sphere of $T_{p}\h ^n$ onto $\partial_{\infty} \h ^n$.

(A4) with this topology, $\overline{\h ^n}$ is a compactification of $\h ^n$ \cite{GL}.

Using this topology, one can define the boundary at infinity of a subset $A$ of $\h ^n$.
Actually, we denote $\overline{A}^\infty$ the closure in $\overline{\h ^n}$ with the ideal
topology. Then $\partial_{\infty}A$ denotes the boundary at infinity of $A$, that is,
$\partial_\infty A=\overline{A}^\infty\cap \partial_\infty \h^n$. Also, denote by
$\mathrm{int}(\cdot)$ the interior of a given set of points.


\subsection{Some models}

\subsubsection*{Poincar\'{e} Ball Model}

There are several models for the hyperbolic space. Among them, the Poincar\'e Ball Model is
very interesting to visualize the hyperbolic geometry.

The Poincar\'{e} Ball model is $(\mathbb B ^{n} , g_{-1})$, where $\mathbb B ^{n}$ is the
Euclidean unit ball in $\R ^{n}$ and $g_{-1}$ is the Poincar\'{e} metric, which is given
at a point $x = (x_1, \ldots , x_n) \in \mathbb B ^n$ by 
\begin{equation}\label{Eq:Metric}
(g_{-1})_x := \frac{4}{(1-|x|^2)^2}\big(\sum _{i=1}^n dx_i ^2\big) ,
\end{equation}here $|\cdot |$ denotes the Euclidean norm. It is well known that, in this
model, the compactification $\overline\h^n$ identified with the closed unit ball and its
ideal boundary corresponds to
$$ \partial _\infty \mathbb H ^{n} = \partial \mathbb B ^{n} = \mathbb S ^{n-1} .$$

In this model, the geodesics are circle arcs (or segment) in $\mathbb{B}^n$ orthogonal to
$\mathbb S^n$. As a consequence, totally geodesic submanifolds are given by spherical caps
(or planar caps) meeting orthogonally the boundary of $\mathbb B^n$. Actually, more
generally, totally umbilical submanifolds of $\h^n$ are given by the intersection of totally
umbilical submanifolds in $\R^n$ with $\mathbb B^n$. 

For example, one can observe that for any point $p \in \R^n \setminus \overline{\mathbb
B^n}$ there exists a unique sphere $S_p$, whose radius is given by $r_p := |p|^2-1$,
that meet orthogonally $\partial \mathbb B ^n$. Hence, $P:=S_p \cap \mathbb B^n$ is a
totally geodesic hyperplane.

From \eqref{Eq:Metric}, one can see that the isometry group of $\h ^n$ is given by the
group of conformal transformations of $\R ^n$ that preserves $\mathbb B ^n$. In
particular, linear isometries are isometries of the model. Euclidean reflections through
hyperplanes containing the origin or inversions through spheres meeting orthogonally the
boundary of $\mathbb B ^n$ are then reflections with respect to totally geodesic
hyperplanes.

Before we continue, let us recall the relation between isometries of the Hyperbolic Space
$\mathbb H ^{n}$, ${\rm Iso}(\h ^n)$, and conformal diffeomorphisms on the sphere at
infinity $\mathbb S ^{n-1}$, ${\rm Conf}(\s^{n-1})$. Using the Poincar\'e ball model, an
isometry $\mathcal I \in {\rm Iso}(\mathbb H ^{n})$ induces a unique conformal
diffeomorphism $\Phi \in {\rm Conf}(\mathbb S ^{n-1} )$; actually this map is bijective. 

\subsubsection*{Halfspace Model}

Another useful model is the Halfspace Model: it is $\R^{n-1}\times(0,+\infty)$ endowed with the metric
\begin{equation}
\frac1{x_n^2}\big(\sum_{i=1}^n d x_i^2\big)
\end{equation}

Let $s=(0,\cdots,0,-1)\in\R^n$ then the map
$$
\Phi: x\mapsto 2\frac{x-s}{|x-s|^2}+s
$$
is conformal and realizes a bijection from $\mathbb B^n$ onto $\R^{n-1}\times(0,+\infty)$.
Actually $\Phi$ is an isometry between the Poincar\'e ball model and the halfspace model.
So properties of this model can be deduced from the preceding one using $\Phi$. For
example, the ideal boundary $\partial_\infty \h^n$ is identified with
$(\R^{n-1}\times\{0\})\cup\{\infty\}$ where $\infty$ is some point added in order to compactify
$\R^{n-1}\times\{0\}$. $\infty$ correspond to $s$ through $\Phi$.


\subsection{Submanifolds of $\h^n$}\label{sec:submanifd}

\subsubsection*{Totally geodesic hyperplanes}

A totally geodesic hyperplane can be characterized by its boundary at infinity. If $P$ is
such a totally geodesic hyperplane, the set $E=\partial_\infty P\subset \partial_\infty
\h^n$ is called an asymptotic equator and for any asymptotic equator $E$ there is a unique
totally geodesic hyperplane $P$ with $\partial_\infty P=E$.

In the Poincar\'e ball model, the asymptotic equators are the hyperspheres of $\mathbb
S^{n-1}$: that is, given any point $x \in \s ^{n-1}$ and radius $r \in (0 , \pi)$, the
submanifold $\partial B_{\mathbb S ^{n-1}}(x, r)\subset \mathbb S^{n-1}$ where $B_{\mathbb
S ^{n-1}}(x, r)$ is the geodesic ball in $\mathbb S ^{n-1}$ centered at $x$ of radius $r
\in (0,\pi)$. In particular, a classical equator centered at $x$, $E(x)$, appearing when $r = \pi
/2$ is an asymptotic equator. When $x$ is the north pole $(0,,\cdots,0,1)$, the hyperplane $P$ associated to $E(x)$ is given by
$$ P(0) := \{ (x_1 , \ldots , x_{n}) \in \mathbb B ^{n} \, : \, \, x_{n} =0 \}.$$

The equidistant hypersurfaces at distance $c$ to some totally geodesic hyperplane $P$ are given by
$$
P_c = \{ {\rm exp }_p (c \, N (p)) \in \h^{n} \, , \, \, p \in P\} ,
$$
where ${\rm exp}$ is the exponential map in $\h^n$ and $N$ is the unit normal along $P$.
Each $P_c$ is totally umbilic with constant principal curvatures $-\tanh (c)$, that
is, 
$$ II_{P_c} = -\tanh (c) I_{P_c},$$
where $I_{P_c}$ and $II_{P_c}$ denote the First and Second Fundamental Form respectively.
Here the orientation for $P_c$ is the one given to be coherent with the normal $N$ on
$P=P_0$. Actually, these equidistant hypersurfaces are the umbilic hypersurfaces of $\h^n$ with
principal curvatures in the interval $(-1,1)$.

Since $P_c$ is at a bounded distance from $P$, it implies that $\partial_\infty P_c=\partial_\infty P$.

\subsubsection*{Horospheres}

Now, based on the above brief introduction, we can define Busemann functions and then
horospheres. Given a unit vector $v $ in $ T\h ^n$, let $\gamma_{v}(t)$ be the oriented
geodesic on $\h^n$ satisfying $\gamma'_{v}(0)=v$, then the Busemann function $B_{v}:\h
^n\rightarrow\mathbb{R}$, associated to $v$, is defined by
\begin{eqnarray*}
B_{v}(p)=\lim\limits_{t\rightarrow+\infty}d(p,\gamma_{v}(t))-t.
\end{eqnarray*}

It is not difficult to see that this function has the following properties (cf. \cite{pe}):

(B1)  $B_{v}$ is a $C^2$ convex function on $\h ^n$;

(B2) the gradient $\nabla B_{v}(p)$ is the unique unit vector $-w$ at $p$ such that
$v(\infty)=w(\infty)$;

(B3) if $w$ is a unit vector such that $v(\infty)=w(\infty)$, then $B_{v}-B_{w}$ is a
constant function on $\h ^n$.

Given a unit vector $v$ in $T\h ^n$, denote by $x$ the point
$v(\infty)\in\partial_{\infty} \h ^n$. The horospheres based or centered at $x$ are
defined to be the \emph{level sets} of the Busemann function $B_{v}$. By (B3), the
horospheres at $x$ do not depend on the choice of $v$. The horoballs based at $x$ are
defined to be the \emph{sublevel sets} of the Busemann function : $\{p\in\h^n\,|\,
B_v(p)\le t\}$.

The horospheres at $x$ give a foliation of $\h ^n$, and by (B1), we know that each element
of this foliation bounds a convex domain in $\h ^n$ which is a horoball. By (B2),
the intersection between a geodesic $\gamma$ and a horosphere based at $\gamma(+\infty)$
is always orthogonal.

The horospheres are the umbilic hypersurfaces of $\h^n$ with constant principal curvatures
equal to $1$ or $-1$ depending one the choice of orientation. Moreover, the induced metric
on each horosphere is flat so they are isometric to $\R^{n-1}$.

In the Poincar\'e ball model, horospheres at $x\in\s^{n-1}\equiv \partial_\infty \h^n$
are given by spheres internally tangent to $\s^{n-1}$ at $x$. The tangency point $x$ is
the unique point at infinity of the horosphere. 

In the halfspace model, the horosphere based at $\infty$ are the horizontal hyperplanes
$\{x_n=c\}$ for $c>0$.


\subsection{Isometries}

Here, we will recall some important properties of the isometry group of $\h ^n$. We
already know that this group acts as the group of conformal transformations of $\R^n$ that
preserves $\mathbb B^n$. So one important fact is that $\mathrm{Iso}(\h^n)$ acts simply
transitively on the space of orthonormal bases of $T\h^n$; more precisely, if ${\bf B}_p$ and
${\bf B}_q$ are orthonormal bases of $T_p\h^n$ and $T_q\h^n$ respectively, then there is one and only one isometry of $\h^n$ sending $p$ to $q$ and ${\bf B}_p$ to ${\bf B}_q$. This property tells that, in a model, we can often assume that we are in some standard position. For example,
if $P$ is a totally geodesic hyperplane, we can assume that it is $\{x_n=0\}\subset \mathbb
B^n$ in the Poincar\'e ball model.

\subsubsection*{Reflections}

Let $P$ be a totally geodesic hyperplane of $\mathbb{H}^{n}$ with unit normal $N$. The
isometry $\boR_P$ fixing points in $P$ and sending $N$ to $-N$ is called the {\bf
reflection through $P$}. We have $\mathcal R _P\circ \mathcal R _P= {\rm Id}$, here ${\rm
Id}$ denotes the identity map. It is important to remark that the group ${\rm Iso}(\h ^n)$
is generated by reflections. 

Let $\Omega$ be a (bounded or unbounded) connected domain in $\h^n$ and
$\boR_{P}$ be the reflection through $P$. We say that $\Omega $ is symmetric with respect
to $P$ if $\boR_P (\Omega) = \Omega$.

\begin{defn}\label{Def:SymmetricP}
Let $P$ be totally geodesic hyperplane in $\h ^n$ and $\Omega$ a domain symmetric with
respect to $P$. A $C^2$ function $u : \Omega \to \mathbb R$ is {\bf symmetric with respect
to $P$} if
$$ u(p) = u(\mathcal R _P (p)) \text{ for all } p \in \Omega .$$
\end{defn}

\subsubsection*{Rotations}

Let $\beta$ be a geodesic. An isometry that preserves the orientation and fix all points
in $\beta$ is called a \textbf{rotation of axis $\beta$}. Actually, the set of rotations
around $\beta$ is a group isometric to $SO_{n-1}(\R)$ (the isomorphism is defined by
looking at the action of a rotation on  the orthogonal to $\beta'(0)$ in $T_{\beta(0)}\h^n$). So
we have a parametrization $\{\mathscr{R}_{\theta} ^{\beta}\}_{\theta\in SO_{n-1}(\R)}$ of
the group of rotations around $\beta$.

Moreover, one can check that any rotation around $\beta$ can be written as the
composition of an even number of reflections with respect to hyperplanes that contain
$\beta$.

If $n\ge 3$, we say that $\Omega $ is axially symmetric with respect to $\beta$ if $\mathcal R _\theta
^{\beta} (\Omega) = \Omega$ for all $\theta \in SO_{n-1}(\R)$. When $n=2$, we say that
$\Ome$ is axially symmetric with respect to $\beta$ if it is symmetric with respect to
$\beta$. Hence, we can define

\begin{defn}\label{Def:AxialSymmetric}
Let $\beta$ be complete geodesic in $\h ^n$ and $\Omega$ a domain axially symmetric with
respect to $\beta $. A $C^2$ function $u : \Omega \to \mathbb R$ is {\bf axially symmetric
w.r.t. $\beta$} if
$$ u(p) = u(\mathcal R^\beta _\theta (p)) \text{ for all } \theta \in SO_{n-1}(\R)\text{
and }p\in\Ome.$$

When $n=2$, $u$ is axially symmetric w.r.t $\beta$ if $u(p)=u(\mathcal R _\beta (p))$ for all $p
\in \Omega$, where $\mathcal R_\beta \in {\rm Iso}(\mathbb H^2)$ is the reflection that
leaves invariant $\beta$.
\end{defn}

\subsubsection*{Hyperbolic Translations}

If $\gamma:\R\to \h^n$ is a unit length geodesic and $t$, the {\bf hyperbolic translation along
$\gamma$ at distance $t$} is the isometry $\boL_\gamma^t$ such that
$\boL_\gamma^t(\gamma(s))=\gamma(s+t)$ and which acts by parallel transport along $\gamma$ on the tangent space.

If $\gamma$ is the geodesic joining the points $x$ and $y$ in $\partial_\infty\h^n$, the
conformal diffeomorphism induced by $\boL_\gamma^t$ on $\mathbb S^{n-1}=\partial_\infty\h^n$ fixes $x$ and $y$. Given any point $p \in \h ^n \setminus \gamma$, the orbit $\{ \mathcal L _{\gamma} ^t (p) \} _{t \in \R}$ is given by an equidistant curve $\gamma _c$ to $\gamma $ passing through $p$, where $c= {\rm dist}(p,\gamma)$.

Let $H_x(s)$ be the horosphere based at $x$ passing by $\gamma(s)$, then $\boL_\gamma^t(H_x(s))=H_x(s+t)$.

Let $P \subset \h ^n$ be a totally geodesic hyperplane and denote $E=\partial _\infty
P$. Let $P_{c}$ be the equidistant hypersurface to $P $ at distance $c$. Let $P_c^+$ and
$P_c^-$ be the two connected component of $\h ^n \setminus P_c$. Note that for any
complete geodesic $\gamma$ contained in $P$, i.e., $\gamma : \R \to P \subset \h ^n$, we have that
$\mathcal L _\gamma ^t (P_c^+) = P_c^+$. In other words, $\mathcal L_\gamma ^t$ leaves
$P_c^+$ invariant (the same is true for $P_c$ and $P_c^-$) for all $\gamma $ contained in
$P$ and for all $t$. This motivates:

\begin{defn}\label{Def:Translating}
A $C^2$ function $u : P_c^+ \to \mathbb R$ is {\bf translating invariant respect to $P$} if
$$
u(p) = u(\mathcal L^t _\gamma (p)) \text{ for all } \gamma \subset P,\ t \in \mathbb R
\text{ and }p\in\Ome,
$$
here $\gamma$ denotes a complete geodesic contained in $P$.
\end{defn}

Note that we could have considered the domain $P_c^-$ in the above definition.
Nevertheless, it is clear that the definition is analogous. Moreover, one can check that
given any complete geodesic $\gamma $ in $\mathbb H ^{n}$ and fixing $t\in \R$, there
exists two totally geodesic hyperplanes $P_1$ and $P_2$, both orthogonal to $\gamma$,
whose associated hyperbolic reflections $\mathcal R _1 ,\mathcal R _2 \in {\rm
Iso}(\mathbb H ^n)$ satisfy
\begin{equation}\label{Eq:TransHyperbolic}
\mathcal L ^\gamma _t = \mathcal R _1 \circ \mathcal R_2 .
\end{equation}

\subsubsection*{Parabolic Translations}

Given any point at infinity $x \in \partial _\infty \mathbb H ^n $, the {\bf parabolic
translations based at $x$} are the isometries of $\h^n$ that acts as Euclidean translations
on each horosphere $H_x$ based at $x$ for the induced Euclidean structure on $H_x$. As a
consequence the subgroup of parabolic translations is isomorphic to $\R^{n-1}$; we have the
parametrization $\{\mathcal T^x _{v} \}_{v\in \mathbb R^{n-1}}$. If $\{H_x(t)\}_{†\in\R}$
is the foliation of the horospheres based at $x$ we have
$$
\mathcal T ^x_v (H_x (t)) = H_x (t) \text{ for all } v\in \mathbb R ^{n-1} \text{ and } t\in \mathbb R.$$

If $P_1$ and $P_2$ are two totally geodesic hyperplanes such that $\partial_\infty P_1\cap
\partial_\infty P_2=\{x\}$ and $\boR_1$ and $\boR_2$ are the reflections with respect to
these hyperplanes then there is some $v\in\R^{n-1}$ such that
\begin{equation}\label{Eq:TransParabolic}
\mathcal R _1 \circ \mathcal R_2=\mathcal T ^x_v .
\end{equation}
Reciprocally, any parabolic translation $\mathcal T_v^x$ can be decomposed in this way.

If $H_x$ is some horosphere based at $x$, we denote by $H_x^+$ and $H_x^-$ the two
connected components of $\h^n\setminus H_x$ such that $\partial_\infty H_x^-=\{x\}$ and
$\partial_\infty H_x^+=\partial_\infty\h^n$; $H_x^-$ is the horoball bounded by $H_x$.
Then for any $v\in \R^{n-1}$, we have $\mathcal T_v^x(H_x^+)=H_x^+$ and $\mathcal T_v^x(H_x^-)=H_x^-$.

\begin{defn}\label{Def:HoroSymmetric}
A $C^2$ function $u : H_x^\pm \to \mathbb R$ is {\bf horospherically symmetric} if
$$ u(p) = u(\mathcal T^x _v (p)) \text{ for all } v \in \mathbb R^{n-1}\text{ and }p\in\Ome .$$
\end{defn}


\section{Symmetry properties of exterior domains}\label{sec:mpm}


\subsection{An important remark}

In this subsection, we give an important result which is the cornerstone of the usage of
the moving plane method in the next subsections.

In order to obtain symmetry conclusions, we must verify that the first PDE in the OEP
\eqref{1.4} is invariant under reflections of $\mathbb H^{n}$. Since
reflection generates ${\rm Iso}(\h ^n)$ by composition, it must be invariant under the
group ${\rm Iso}(\h ^n)$. Invariant means that, if $u$ is a solution to \eqref{1.4} in
$\Omega$, and $\mathscr{I} : \mathbb H ^{n } \to \mathbb H ^{n}$ an
isometry, then $v (p) = u(\mathscr{I}(p))$ is a solution to \eqref{1.4} in
$\tilde \Omega = \mathscr{I}^{{-1}}(\Omega)$.

Let $P$ be a totally geodesic hyperplane of $\mathbb{H}^{n}$ and $\mathscr{R}_{P}$  the
reflection through $P$. Let $\Omega$ be a (bounded or unbounded) connected domain in
$\mathbb{H}^{n}$. We denote by $\Omega^+$ the subset $\Omega \cap P^{+}$ (where $P^+$ is
one connected component of $\h^n\setminus P$), that we
assume to be nonempty, and denote by $\widetilde\Omega^+$ its reflection through $P$, \textit{i.e.} $
\widetilde\Omega^+ = \mathscr{R}_{P} ( \Omega^+ )$. Define a function $v(p)$ as follows
\begin{eqnarray}  \label{ARF1}
v(p)=u(\mathscr{R}(p)) \text{ for } p \in \widetilde\Omega^+  .
\end{eqnarray}

For the function $v$, we can prove the following.

\begin{lemma}[\cite{EM}] \label{lemmaARF}
The function $v(p)$ defined by (\ref{ARF1}) satisfies the first PDE in the OEP \eqref{1.4}.
\end{lemma}


\subsection{A maximum principle}

In order to apply the Moving Plane Method, we need the maximum principle at infinity given by the following lemma.

\begin{lemma}\label{lem:maxp}
Let $\Ome$ be a connected domain in $\h^n$ and $E$ be an asymptotic equator. Let $w\in
C^2(\Ome)$ be a bounded below solution of
$$
\Delta w+cw=0
$$
satisfying $\liminf w(p)\ge 0$ when $p$ converges to some point in $\partial\Ome\cup
(\partial_\infty\Ome\setminus E)$.

If $c\le 0$ in $\Ome$, then either $w\equiv 0$ or $w>0$ in $\Ome$.
\end{lemma}

We notice that $c$ is only assumed to be a measurable function in $\Ome$.

\begin{proof}
It is enough to prove $w\ge 0$.

In the ball model, we can assume that $E=\s^{n-1}\cap\{x_n=0\}$. On $\s^{n-1}$, we define
$\lambda(x)=|x_n|^{-1/2}$ which is integrable on $\s^{n-1}$. Let $u$ be the harmonic
extension of $\lambda$ to $\h^n$. $u$ is then positive (actually $u\ge 1$) and $u(p)\to \infty$ as $p$
approaches $E$. Thus, for $t$ positive, we have $\liminf_{p\to
\partial\Ome\cup\partial_\infty\Ome}(w+t u)\ge t$. Moreover 
$$
\Delta (w+tu)=-cw\le -c(w+tu)
$$
So the maximum principle implies that $w+tu\ge 0$ on $\Ome$. As it is true for any $t>0$,
$w\ge 0$.
\end{proof}


\subsection{Moving Plane Method and symmetries of domains}

In this subsection, we apply the moving plane method to obtain a symmetry result for some
$f$-extremal domains.

When $P$ is a totally geodesic hyperplane in $\h^n$ and $\gamma$ is a geodesic such that
$\gamma(0)\in P$ and $\gamma$ is normal to $P$, we define a foliation of $\h^n$ in the
following way : let $P(t)$ be the totally geodesic hyperplane passing through $\gamma(t)$
and normal to $\gamma$. With this construction, we have the following symmetry result.

\begin{theorem}\label{th:mpm}
Assume that $U$ is an open domain in $\mathbb{H}^{n}$ (non necessarily connected), with
$C^{2}$ boundary $\Sigma$, such that $\partial _{\infty} U \subset E $,
where $E$ is an asymptotic equator at the boundary at infinity $\partial _{\infty}\mathbb
H^{n}$. Let $P$ be the totally geodesic hyperplane whose boundary at
infinity is $E$, i.e., $\partial _{\infty} P = E$ and let $\gamma$ be a geodesic normal to
$P$. Let $\{P(t)\}_{t\in\R}$ be the associated foliation.

Assume that the domain $\Omega = \h ^n \setminus \overline{U}$ is connected and the OEP
\eqref{1.4} has a solution $u\in{C}^{2}(\overline{\Omega})$, with $f$
satisfying~\eqref{hyp:h2}. Then, there is $t_0\in\R$ such that $\Omega$ is invariant by
the reflection $\boR_{P(t_0)}$ i.e., $\boR_{P(t_0)} (\Omega) =\Omega$, and $u$ is also
invariant under $\boR_{P(t_0)}$, that is, $u(p) = u(\boR_{P(t_0)} (p))$ for all $p\in
\Omega$.
\end{theorem}

We first remark that if $\partial_\infty U\neq \emptyset$, $t_0$ is necessarily $0$. The
second remark is that, since $\Ome$ is connected and $\partial_\infty U\subset E$,
$\partial_\infty \Ome=\partial_\infty \h^n$.

\begin{proof}
For $t\in\R$, we denote by $\boR_t$ the reflection through $P(t)$. We also denote by
$P^-(t)$ (resp. $P^+(t)$) the open halfspace bounded by $P(t)$ that contains
$\{\gamma(s),s<t\}$ (resp. $\{\gamma(s),s>t\}$). We then introduce $U_t^-=P^-(t)\cap U$, $U_t^+=P^+(t)\cap U$,
$\Ome_t^-=P^-(t)\cap \Ome$ and $\Ome_t^+=P^+(t)\cap \Ome$. We also define
$\widetilde\Ome_t^+=\boR_t(\Ome_t^+)\subset P^-(t)$ and $\widetilde U_t^-=\boR_t(U_t^-)$
(see Figure~\ref{fig:fig4}).

On $\widetilde\Ome_t^+$, the function $v_t=u\circ \boR_t$ is defined and solves the PDE in
\eqref{1.4}. The first important fact is the following.

\begin{figure}
\begin{center}
\resizebox{0.6\linewidth}{!}{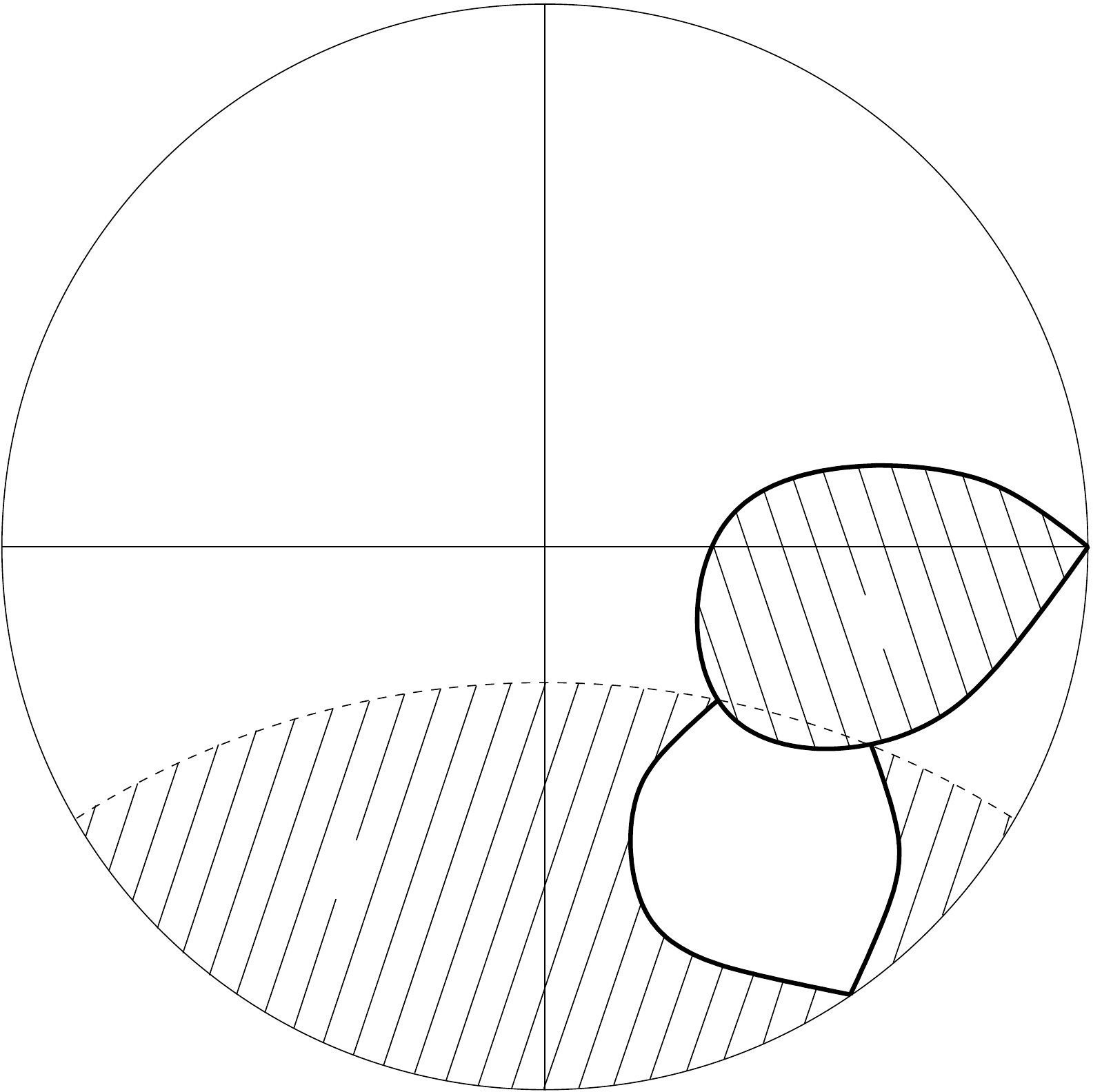}
\end{center}
\caption{The domain $\widetilde\Ome_t^+$}\label{fig:fig4}
\end{figure}

\begin{fact}
Let $t$ be non positive. If $\widetilde\Ome_t^+\subset \Ome_t^-$, then $v_t\le u$ on
$\widetilde\Ome_t^+$. Moreover if $v_t(p)=u(p)$ at some point $p\in\widetilde\Ome_t^+$
then $\widetilde\Ome_t^+=\Ome_t^-$ ($\Ome$ is symmetric with respect to $P_t$) and $v_t=u$
on $\widetilde\Ome_t^+=\Ome_t^-$.
\end{fact}
So, let us assume $\widetilde\Ome_t^+\subset \Ome_t^-$ and  let $w_t$ be $u-v_t$ on
$\widetilde\Ome_t^+$. $v_t$ satisfies the following conditions
$$
\begin{cases}
\Delta v_t+f(v_t)=0&\textrm{in } \widetilde\Ome_t^+,\\
v_t(p)=u_t(p)& \textrm{if } p\in \partial\widetilde\Ome_t^+\cap P(t),\\
v_t(p)=0& \textrm{if } p\in \partial\widetilde\Ome_t^+\cap P^-(t),\\
\langle\nabla v_t,\vec\nu\rangle=\alpha &\textrm{on }\partial\widetilde\Ome_t^+\cap P^-(t)
\end{cases}
$$
where $\partial\widetilde\Ome_t^+\cap P^-(t)$ is included in $\boR_t(\partial\Ome)$.

As a consequence, the function $w_t$ solves the PDE
$$
\Delta w_t+cw_t=0
$$
where $c$ is defined by
$$
c(p)=\begin{cases}
-1&\textrm{if } w_t(p)=0\\
\frac{f(u(p))-f(v_t(p))}{u(p)-v_t(p)}&\textrm{if } w_t(p)\neq 0
\end{cases}
$$
Since $f$ is non increasing, $c$ is a non positive function. 

Since $u$ is bounded, $w_t$ is bounded too. Let $(p_n)$ be a sequence of points that
converges to some point
$q\in\partial\widetilde\Ome_t^+\cup\partial_\infty\widetilde\Ome_t^+$. Let us study the behaviour of the sequence
$(w_t(p_n))_{n\in \mathbb N}$.

First $q$ could be in $\partial\widetilde\Ome_t^+$, since $u\ge 0$ in $\overline\Ome$, we
have $\lim w_t(p_n)=w_t(q)\ge 0$ ($w_t(q)=0$ on $\partial\widetilde\Ome_t^+\cap P(t)$ and
$w_t(q)\ge 0$ on $\partial\widetilde\Ome_t^+\cap P^-(t)$). Let us assume now that
$q\in\partial_\infty\widetilde\Ome_t^+$. The first case is $\lim
d(p_n,\boR_t(\partial\Ome))=+\infty$, this implies that $\lim d(p_n,\partial\Ome)=+\infty$
and $\lim d(\boR_t(p_n),\partial\Ome)=+\infty$; thus, $u$ having a limit far from the boundary, $\lim
w_t(p_n)=0$. The last possibility is $d(p_n,\boR_t(\partial\Ome))$ stays bounded. This
case can only appear if $q\in \boR_t(E)$ so it does not matter in order to apply Lemma~\ref{lem:maxp}.
Hence $w_t$ satisfies to the hypotheses of Lemma~\ref{lem:maxp} and $w_t\ge 0$ : $v_t\le u$.

The equality case follows easily. This finishes the proof of the above fact.

We are now ready to apply the moving plane method. First since $\partial_\infty U\subset
E$, there is $T\le 0$ such that for any $t\le T$, $P^-(t)\subset \Ome$ and let $t_1$ be
the largest non positive number $T$ such that this property is true. We assume for the
moment that $t_1<0$.

Since $\partial_\infty U\subset E$, for any $t<0$, $\overline{P^-(t)}\cap \overline U$ is
compact. This implies that $P(t_1)$ is tangent to $\partial\Ome$. Since $\partial\Ome$ is
$C^2$, there exists $\epsilon>0$ such that $\widetilde U_t^-\subset U_t^+$ (or $\widetilde
\Ome_t^+\subset \Ome_t^-$) and $\partial \Ome$ is not orthogonal to $P(t)$ for $t\in(t_1,t_1+\epsilon)$. Looking at the first non positive time where such properties stop to be true, one of the following situations will happen:

\begin{itemize}
\item[(A)] There exists $\bar t \in (t_{1 } ,0)$ such that $\widetilde \Omega _{\bar
t}^{+}$ is internally tangent to the boundary of $\Omega _{\bar t}^{-} $ at some point
$\bar p$ not in $P(\bar t)$ and $\widetilde \Omega _{ t}^{+} \subset \Omega _{ t}^{+}$ for
all $t\in (-\infty,\bar t]$.

\item[(B)] There exists $\bar t \in (t_{1 } ,0)$ such that $P(\bar t)$ arrives at a
position where $\partial\Ome$ is orthogonal to $P(\bar t)$ at $\bar p \in \partial\Ome\cap
P(\bar t)$ and $\widetilde \Omega _{ t}^{+} \subset \Omega _{ t}^{+}$ for all $t\in
(-\infty,\bar t]$.

\item[(C)] $\widetilde \Omega _{ t}^{+} \subset \Omega _{ t}^{+}$ for all $t\in
(-\infty,0]$ (this case corresponds also to the case $t_1=0$).
\end{itemize}

If (A) occurs, the function $w_{\bar t}$ is non negative by the above fact and, at $\bar p$,
we have $w_{\bar t}(\bar p)=0$ and $\partial_\nu w_{\bar t}(\bar p)=0$ (this second
property comes from the fact that $v_{\bar t}$ has the same Neumann data as $u$ on
$\partial\widetilde\Ome_t^+\cap P^-(t)$). So applying the Hopf boundary maximum principle
to $w_{\bar t}$ yields 
$u-v_{\bar t}\equiv 0$ in $\widetilde \Ome_t^+$, which implies that $\widetilde\Ome_{\bar
t}^+=\Ome_{\bar t}^-$ so $\boR_{\bar t}(\Ome)=\Ome$.

If (B) occurs, the point $\bar p$ is then a right angle corner point of $\partial
\widetilde\Ome_{\bar t}^+$ so the boundary maximum principle cannot be applied directly
(the requisite of the interior tangent ball is not available). Nevertheless, we can apply
Serrin's Corner Lemma to obtain that $P(\bar t)$
must be a hyperplane of symmetry. Serrin's Corner Lemma appears first as Lemma~2 in
\cite{s}, however this version is not sufficient in our situation; so we use an improved
version as in the proof of \cite[Theorem 8.3.2, p. 145]{ps} (see also the discussion in
\cite[Appendix to Section~8.3, p. 149-151]{ps}). We also refer to \cite{EM,mr} for a
detailed exposition in the hyperbolic setting. 

So we are let with case (C) where we get $\widetilde\Ome_0^+\subset \Ome_0^-$. Then we can
do the same argument as above but looking at $t>0$ and exchanging the role played by
$\Ome_t^+$ and $\Ome_t^-$. This will give us either that $\Ome$ is symmetric with respect
to some $P(\bar t)$ for $\bar t>0$ or $\boR_0(\Ome_0^-)\subset \Ome_0^+$. This last
inclusion with $\widetilde\Ome_0^+\subset \Ome_0^-$ gives us
$\widetilde\Ome_0^+=\Ome_0^-$: $\Ome$ is symmetric with respect to $P(0)$.

The symmetry of the function $u$ then comes by applying the above fact.
\end{proof}


\subsection{Applications}

Now, we can apply Theorem \ref{th:mpm} to obtain classification results of some
$f$-extremal domains defined as exterior domains.

\begin{theorem} \label{CorBounded}
Assume that $U$ is a bounded open domain in $\mathbb{H}^{n}$, with $C^{2}$ boundary, $
\Omega = \h^n \setminus \overline{U} $ is connected and on which the OEP \eqref{1.4} has a
solution $u\in{C}^{2}(\overline{\Omega})$ with $f$ satisfying~\eqref{hyp:h2}. Then
$U$ must be a geodesic ball and $u$ is radially symmetric.
\end{theorem}

This theorem is similar to Theorem~1 in \cite{Rei} by Reichel.

\begin{remark}\label{rk:1}
In the above theorem, the function $u$ is then a function of the distance $s$ to some
point $p_0\in\h^n$. The PDE in~\eqref{1.4} can then be written in term of the variable $s$
as the ODE:
$$
\partial_s^2u+(n-1)\mathrm{cotanh}(s)\partial_s u+f(u)=0.
$$
\end{remark}

Next, we will classify exterior domains $\Omega=\h^n\setminus \overline U$ when
$\partial_\infty U $ has only one point.

\begin{theorem}\label{OnePoint}
Assume that $U$ is a domain in $\mathbb{H}^{n}$, with $C^2$ boundary and whose asymptotic
boundary is a point $x_{0} \in \partial _\infty \mathbb H ^n$. Assume $\Omega = \h ^n
\setminus \overline{U}$ is connected and on which the OEP \eqref{1.4} has a solution
$u\in{C}^{2}(\overline{\Omega})$  with $f$ satisfying~\eqref{hyp:h2}.

Then, $\Omega$ is the exterior of a horoball at $x_0$ and $u$ is horospherically
symmetric.
\end{theorem}

This result can be compare to Theorem~C in~\cite{SETou} by Sa Earp and Toubiana. We can
also think to Theorem~A in~\cite{dCL} by do Carmo and Lawson about the geometry of
constant mean curvature hypersurfaces in $\h^n$.

\begin{proof}
Let $P$ be a totally geodesic hyperplane such that $x_0\in\partial_\infty P$ and apply
Theorem~\ref{th:mpm} with $E=\partial_\infty P$. Since $\partial_\infty U\neq \emptyset$,
the remark below Theorem~\ref{th:mpm} implies that $\Ome$ and $u$ are symmetric with
respect to $P$.

Let $\mathcal T^{x_0}_v$ be a parabolic translation based at $x_0$ and $P_1$ and $P_2$ be
two totally geodesic hyperplanes such that $\partial_\infty P_1\cap \partial_\infty
P_2=\{x_0\}$ and $\mathcal T^{x_0}=\boR_{P_1}\circ\boR_{P_2}$. Since $\Ome$ and $u$ are
invariant by $\boR_{P_1}$ and $\boR_{P_2}$ they are invariant by $\mathcal T^{x_0}_v$.

Thus if $p\in\partial\Ome$, $\mathcal T_v^{x_0}(p)\in\partial\Ome$ for any $v\in\R^{n-1}$
and each
connected component of $\partial\Ome$ is a horosphere based at $x_0$. Now since $\Ome$ is
connected and $\partial_\infty U=\{x_0\}$, $U$ is a horoball. Finally $u$ is
horospherically symmetric (see \cite{dCL} for similar arguments).
\end{proof}

\begin{remark}\label{rk:2}
In the above result, the function $u$ is then a function of the distance $s$ to some
horospheres in $\h^n$. The PDE in~\eqref{1.4} can then be written in term of the variable
$s$ as the ODE:
$$
\partial_s^2u-(n-1)\partial_s u+f(u)=0
$$
\end{remark}

Also, another consequence of Theorem \ref{th:mpm} and Definition \ref{Def:AxialSymmetric}
is the following:

\begin{theorem}\label{TwoPoints}
Assume that $U$ is a domain in $\mathbb{H}^{n}$, with boundary a $C^2$
hypersurface $\Sigma$ and whose asymptotic boundary consists in two distinct points $x, y
\in \mathbb S ^{n-1}$, $x\neq y$. 

Assume $\Omega = \h ^n \setminus \overline{U}$ is connected and on which the OEP
\eqref{1.4} has a solution $u\in{C}^{2}(\overline{\Omega})$ with $f$
satisfying~\eqref{hyp:h2}. Then $\Omega$
is rotationally symmetric with respect to the axis given by the complete geodesic
$\beta $ whose boundary at infinity is $\{ x,y\}$, i.e., $\beta(\infty)= x$ and $\beta(-\infty) =
y$. In other words, $\Omega$ is invariant by the group of rotations in
$\mathbb H^n$ fixing $\beta$. Moreover, $u$ is axially symmetric w.r.t. $\beta$.
\end{theorem}


\section{Invariance of $f$-extremal domains}\label{sec:inv}

When $E$ is an asymptotic equator in $\partial_\infty \h^n$, the closure (in
$\partial_\infty \h^n$) of each connected component of $\partial_\infty\h^n\setminus E$ is
called an asymptotic hemisphere associated to $E$. If one asymptotic hemisphere $C$ is chosen
and $P$ is the totally geodesic hyperplane with $\partial_\infty P=E$, we consider the
equidistant hypersurfaces $P_c$ to $P$ (see Section~\ref{sec:submanifd} for the
definition) and $P_c^+$ the connected component of $\h^n\setminus P_c$ with
$C=\partial_\infty P_c^+$.

The next result mainly says that a $f$-extremal domain $\Ome$ such that $\partial_\infty\Ome$ is an asymptotic hemisphere is translating invariant with respect to some totally geodesic hyperplane.

\begin{theorem}\label{Th:Graphical}
Let $E$ be an asymptotic equator and $C$ an asymptotic hemisphere associated to $E$; we
also denote by $P$ the totally geodesic hyperplane with $\partial_\infty P=E$. Let $\Ome$
be a connected domain with $C^2$ boundary $\Sigma$ such that $\partial_\infty \Sigma=E$ and
$\partial_\infty \Ome=C$. Assume that the OEP \eqref{1.4} has a solution $u\in
C^2(\overline\Ome)$ with $f$ satisfying~\eqref{hyp:h2}.

Then $\Ome=P_c^+$ for some $c\in\R$ and $u$ is translating invariant respect to $P$.
\end{theorem}

Let us mention that this result has a great similarity with \cite[Theorem~3.1]{dCL} by
do Carmo and Lawson dealing with constant mean curvature hypersurfaces.

\begin{proof}
Let $\gamma$ be a geodesic normal to $P$ such that $\gamma(+\infty)\in C$ and
$\{\mathcal L_\gamma^t\}_{t\in\R}$ the group of hyperbolic translations along $\gamma$.

For any $t\in \R$, we define $\Ome_t=\mathcal L_\gamma^t(\Ome)$ and $v_t=u\circ \mathcal
L_\gamma^{-t}$ which is a function defined on $\Ome_t$. The first part of the proof
consists in proving the following fact
\begin{fact}
For any $t\le0$, $\Ome\subset \Ome_t$ and $u\le v_t$ on $\overline\Ome$.
\end{fact}
First we notice that for $t<0$, $\partial _\infty\mathcal L_\gamma^t(P)\cap C=\emptyset$. So let us prove
that, if $\Ome\subset \Ome_t$, then $u\le v_t$ on $\overline\Ome$.

Since $\mathcal L_\gamma^t$ is an isometry, $v_t$ satisfies:
\begin{equation}\label{Eq:vt}
\begin{cases} 
\Delta{v_t}+f(v_t)=0  &\text{in } \Omega_t,\\
v_t>0  &\text{in } \Omega_t,\\
v_t=0  &\text{on }  \partial\Ome_t=\boL_\gamma^t(\partial\Ome),\\
v_t (p)\to C & \text{uniformly as }  d(p , \partial \Omega_t )\to +\infty \\
\langle\nabla{v_t},\vec{v}\rangle=\alpha &\text{on }\partial\Ome_t,
\end{cases}
\end{equation}

Then the function $w_t=v_t-u$ solves the equation $\Delta w_t+cw_t=0$ where $c$ is defined by
$$
c(p)=\begin{cases}
-1&\textrm{if } w_t(p)=0\\
\frac{f(u(p))-f(v_t(p))}{u(p)-v_t(p)}&\textrm{if } w_t(p)\neq 0
\end{cases}
$$
Since $f$ is non increasing, $c$ is a non positive function. As in the proof of
Theorem~\ref{th:mpm}, $\liminf w_t(p_n)\ge 0$ for any sequence $p_n$ converging to some
point of $\partial\Ome\cup(\partial_\infty\Ome\setminus E)$. So Lemma~\ref{lem:maxp} gives
$w_t\ge 0$ and $v_t\ge u$ on $\Ome$.

Now let us prove $\Ome\subset \Ome_t$ for $t\le 0$. Since $\partial_\infty \Sigma=E$, $\mathcal
L_\gamma^t(\Sigma)$ goes to $\gamma(-\infty)$ as $t$ goes to $-\infty$. Hence there is
$T<0$ such that for any $t\le T$, $\Ome\subset \Ome_t$. Let $\bar t$ be the largest non
positive $T$ satisfying the above property. Since for negative $t$, $\Ome\setminus\Ome_t$
has compact closure in $\h^n$, if $\bar t<0$, we have $\Ome\subset \Ome_{\bar t}$ and
$\partial\Ome$ and $\partial \Ome_{\bar t}$ are tangent at some point $\bar p$. Since
$\Ome\subset \Ome_{\bar t}$, we have $u\le v_{\bar t}$ and both solutions of the PDE have
the same Neumann boundary data at $\bar p$. Thus Hopf's boundary maximum principle implies
that $u=v_{\bar t}$ so $\Ome=\Ome_{\bar t}$ \textit{i.e.} $\bar t=0$. This finishes the
proof of the fact.

The second step of the proof is to prove invariance with respect to hyperbolic translation
along $P$. So let $\tilde\gamma$ be a geodesic in $P$ and $\{\mathcal
L_{\tilde\gamma}^s\}_{s\in\R}$ the group of hyperbolic translations along $\tilde\gamma$. We
then denote by $\mathcal L^{s,t}$ the isometry of $\h^n$ which is the composition
$\mathcal L_{\tilde\gamma}^s\circ\mathcal L_{\gamma}^t$. Then we define
$\Ome_{s,t}=\mathcal L^{s,t}(\Ome)$ and $v_{s,t}=v\circ(\mathcal L^{s,t})^{-1}$ which is
defined on $\Ome_{s,t}$. Then we have the following fact

\begin{fact}
For any $s\in\R$ and $t<0$, $\Ome\subset \Ome_{s,t}$ and $u\le v_{s,t}$ on $\overline\Ome$.
\end{fact}

The proof of this fact is similar to the first one. The arguments to prove that, if
$\Ome\subset \Ome_{s,t}$, $u\le v_{s,t}$ on $\overline\Ome$ are the same, once we have
noticed that $\partial _\infty \mathcal L^{s,t}(P)\cap C=\emptyset$. This comes from the fact that each
connected component of $\partial_\infty \h^n\setminus E$ is stable by the group
$\{\mathcal L_{\tilde\gamma}^s\}_{s\in\R}$ and the inclusion is true for $s=0$.

To prove $\Ome\subset \Ome_{s,t}$, we fix $t<0$ and we first remark that $\Ome\subset
\Ome_t=\Ome_{0,t}$. So we can look at
$$
s^-=\inf\{S\le 0\,|\, \forall s\in[S,0], \Ome\subset\Ome_{s,t}\}\quad \textrm{ and }
\quad s^+=\sup\{S\ge 0\,|\, \forall s\in[0,S], \Ome\subset\Ome_{s,t}\}.
$$
If $s^-$ is finite, as above, $\Ome_{s^-,t}$ and $\Ome$ are tangent somewhere and Hopf's
boundary maximum principle gives a contradiction. The same is true for $s^+$. This
finishes the proof of the fact.

Now we can finish the proof of our theorem. Since $\Ome\subset \Ome_{s,t}$ for $s\in\R$
and $t<0$, letting $t\to 0$ we get, $\Ome\subset \Ome_{s,0}$. Taking the image of this
inclusion by $\boL_{\tilde\gamma}^{-s}$ we get $\Ome_{-s,0}\subset \Ome$. Thus
$\Ome=\Ome_{s,0}$, $\Ome$ is translation invariant along $\tilde\gamma$ and then $P$ since
$\tilde\gamma$ is arbitrary.

The same argument gives that $u=v_{s,0}=u\circ\boL_{\tilde\gamma}^{-s}$; $u$ is
translation invariant along $P$.

This invariance implies that each connected component of $\partial\Ome$ is an equidistant
hypersurface $P_c$. Thus, since $\Ome$ is connected and $\partial_\infty\Ome=C$, we have
$\Ome=P_c^+$ for some $c\in\R$.
\end{proof}

\begin{remark}
In the above result, the function $u$ is then a function of the distance $s$ to some
hyperplane $P$ in $\h^n$. The PDE in~\eqref{1.4} can then be written in term of the variable
$s$ as the ODE:
$$
\partial_s^2u+(n-1)\tanh(s)s\partial_s u+f(u)=0.
$$

We notice that some aspects of the study of this ODE and the ones appearing in
Remarks~\ref{rk:1} and \ref{rk:2} can be found in~\cite{BiMa}.
\end{remark}


\section{$f$-extremal domains in $\mathbb H^2$}\label{sec:dim2}

From now on in this section, we will focus on the two dimensional case, i.e., $\Ome\subset \h^2$.

More precisely, we consider an unbounded open connected domain $\Ome$ in $\h^2$ whose $C^2$
boundary has only one connected boundary component $\Gamma$. We also assume that the
OEP~\eqref{1.4} has a solution $u$ on $\Ome$. If $\Gamma$ is compact and
hypothesis~\eqref{hyp:h2} is assumed,
Corollary~\ref{CorBounded} implies that $\Ome$ is the exterior of a geodesic ball and $u$
is radially symmetric. If $\Gamma$ is unbounded, then, in order to apply results of the
preceding sections, we need to understand the asymptotic behavior of $\Gamma$ : what is
$\partial_\infty\Gamma$?

We notice that, in Lemmata~\ref{lem:RRS}, \ref{lem:G(p)} and \ref{lem:2point} below,
hypothesis~\eqref{hyp:h2} is not assumed on $f$. As in Section~\ref{sec:mpm}, we use the
moving plane method but only for compact parts of domains. So we do not need to assume any
monotonicity for $f$. We refer to \cite{RRS} for such use of the moving plane method.

If $p\in \Gamma$, let us denote by $G(p)$ the endpoint in $\partial_\infty \h^2$ of the inward
normal half-geodesic line to $\Gamma=\partial\Ome$ at $p$. We recall that $(pG(p))$
denotes the half-geodesic line starting at $p\in \mathbb H^{n}$ and ending at $G(p)\in
\partial _{\infty} \mathbb H^{n}$.

The first step of our study of
$\partial_\infty \Gamma$ is given by the following lemma which is similar to
\cite[Lemma~2.4]{RRS}.

\begin{lemma}\label{lem:RRS}
Let $\Ome$ be as above and $p\in\Gamma$. The half-geodesic line $(pG(p))$ is inside $\Ome$.
\end{lemma}

\begin{proof}
Same as in~\cite{RRS}
\end{proof}

The consequence of this property is that $G(p)\in \partial_\infty \Ome$. Actually we
can say more, we have the following lemma.

\begin{lemma}\label{lem:G(p)}
Let $\Ome$ be as above and $p\in \Gamma$. Then, either $G(p)\notin \partial_\infty \Gamma$
or $\Ome$ is a horodisk.
\end{lemma}

\begin{proof}
Let us assume that $G(p)\in \partial_\infty\Gamma$. Using the half-space model for
$\h^2=\{(x,y)\in\R\times(0,+\infty)\}$,
we can assume that $p=(0,1)$ and $G(p)=\infty$, i.e., $\Gamma$ is horizontal at
$p$. $\Gamma\setminus\{p\}$ has two connected components that we denote by $\Gamma_l$ and
$\Gamma_r$ which, near $p$, lies respectively in $\{x>0\}$ and $\{x<0\}$. Moreover we can
assume that $\infty\in\partial_\infty\Gamma_l$. This means that, for any $R>0$, there are
points in $\Gamma_l$ outside the halfdisk $\{(x,y)\in\R\times(0,+\infty)\, : \,x^2+y^2\le
R^2\}$.

Let $\Ome_{l,R}$ be the connected component of $\Ome\setminus(pG(p))\cap
\{(x,y)\in\R\times(0,+\infty) \, : \, x^2+y^2\le R^2\}$ with $p$ and a part of $\Gamma_l$ in
its closure (the connected component that lies in $\{x>0\}$ near $p$). Let $D_t$ be the
half-disk $D_t=\{(x,y)\in \R\times(0,+\infty)\, : \, (x-t)^2+y^2\le t^2\}$. First we see that,
since $\Gamma_l$ has points
outside $\{(x,y)\in\R\times(0,+\infty) :x^2+y^2\le 4R^2\}$, for any $0<t\le R$,
$\Ome_{2R,t}=D_t\cap \Ome_{l,2R}$ is bounded and empty if $t$ is sufficiently small (see
Figure~\ref{fig:fig1}).

\begin{figure}[!h]
\begin{center}
\resizebox{0.7\linewidth}{!}{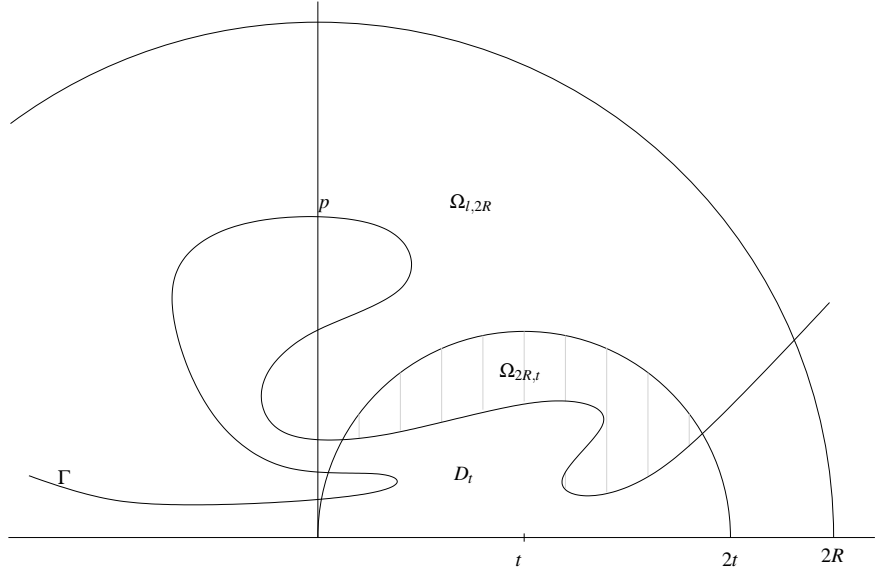}
\end{center}
\caption{The domain $\Ome_{2R,t}$}\label{fig:fig1}
\end{figure}

So now we can do Alexandrov reflection for the subsets $\Ome_{2R,t}$ for $t\le R$, we
symmetrize $\Ome_{2R,t}$ with respect to $\partial D_t$ which is a geodesic of $\h^2$ and
we get $\widetilde\Ome_{2R,t}$. On $\widetilde\Ome_{2R,t}$, there is the symmetrized
solution $\tilde u_{2R,t}$. We fix $R$ and let $t$ move from $0$ to $R$. When $t$ is small
$\widetilde\Ome_{2R,t}\subset \Ome$ and $\tilde u_{2R,t}$ is below $u$. If there is a
first contact between $\tilde u_{2R,t}$ and $u$ for $t\le R$, we get a symmetry for $\Ome$
and $\Ome$ is bounded, which is a contradiction.

So we have $\widetilde \Ome_{2R,t}\subset \Ome$ and $\tilde u_{2R,t}\le u$ for any $t\le
R$. Thus $\tilde u_{2R,R}\le u$ on $\widetilde\Ome_{2R,R}$ for any $R>0$. Letting $R$ goes
to $+\infty$ we get that $\Ome$ and $u$ are symmetric with respect to $\{x=0\}$.

If $(x,y)\in\R^2$, we denote $*(x,y)=(-x,y)$. Let $q$ be a point in $\Gamma$ close to
$p$; we have $G(*q)=*G(q)$. Because of Lemma~\ref{lem:RRS}, the geodesic half-lines
$(qG(q))$ and $(*qG(*q))$ do not intersect $\Gamma$. Joining this two geodesic half-lines
by the piece of arc in $\Gamma$ between $q$ and $*q$, we get a proper curve in
$\R\times\R_+^*$ that does not cross $\Gamma$. If $G(q)\not= \infty$, this implies that
$\Gamma$ stays far away from $\infty$ (see Figure~\ref{fig:fig2}). As we assume that $\infty\in
\partial_\infty\Gamma$, we can conclude that $G(q)=\infty$ for $q$ close to $p$. So the
set $\{q\in \Gamma \, : \, G(q)=\infty\}$ is open and closed in $\Gamma$ and $G(q)=\infty$
for any $q$ in $\Gamma$ and $\Ome$ is a horodisk.

\begin{figure}[!h]
\begin{center}
\resizebox{0.6\linewidth}{!}{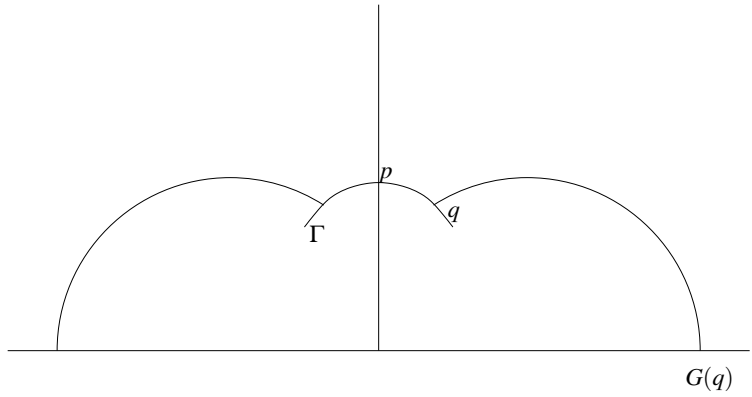}
\end{center}
\caption{The barriers for $\Gamma$ if $G(q)\neq\infty$}\label{fig:fig2}
\end{figure}
\end{proof}

The preceding lemma allows us to control the asymptotic behaviour of $\Gamma$.

\begin{lemma}\label{lem:2point}
$\partial_\infty\Gamma$ is made of at most two points.
\end{lemma}

\begin{proof}
Let $p$ be a point in $\Gamma$, $\Gamma\setminus \{p\}$ has two connected components
$\Gamma_1$ and $\Gamma_2$ and both $\partial_\infty\Gamma_1$ and $\partial_\infty\Gamma_2$
contain at least one point. We want to prove that both are made of only one point.

From Lemma~\ref{lem:G(p)} we know that $\partial_\infty\Gamma\neq \partial_\infty \h^2$
(either $G(p)\in \partial_\infty \h^2\setminus\partial_\infty\Gamma$ or $\Ome$ is a horodisk and
$\partial_\infty\h^2$ is made of one point). Besides $\partial_\infty \Gamma_i$ are both
intervals of $\partial_\infty\h^2$

Let us assume that $\partial_\infty\Gamma_1$ is not reduced to one point. So, in the
half-space model, we can assume that $\infty\notin\partial_\infty\Gamma$, $\partial_\infty
\Gamma_1$ is an interval that contains $[-1,1]\times\{0\}$ and the geodesic $\{x=0\}$ is
transverse to $\Gamma$. We parametrized $\Gamma_1$ by arc-length on $\R_+^*$ and denote by
$(x_{\Gamma_1},y_{\Gamma_1})$ the parametrization. Because of the hypothesis on
$\Gamma_1$, there is an increasing sequence $(s_n)$ such that $\{s_n\}_{n\in\mathbb
N}=x_{\Gamma_1}^{-1}(0)$. We then have $y_{\Gamma_1}(s_n)\to 0$ and the inward unit normal
to $\Gamma_1$ points downward at $(x_{\Gamma_1},y_{\Gamma_1})(s_n)$ when $n$ is even or
$n$ is odd (depending on the unit normal at $(x_{\Gamma_1},y_{\Gamma_1})(s_0)$) (see
Figure~\ref{fig:fig3}). Let us assume that it is the case when $n$ is even. Choose $k$ large so
that $y_{\Gamma_1}(s_{2k})<1$, then $G((x_{\Gamma_1},y_{\Gamma_1})(s_{2k}))\in[-1,1]
\times\{0\}\subset\partial_\infty\Gamma$ which is a contradiction with
Lemma~\ref{lem:G(p)}.

\begin{figure}[!h]
\begin{center}
\resizebox{0.8\linewidth}{!}{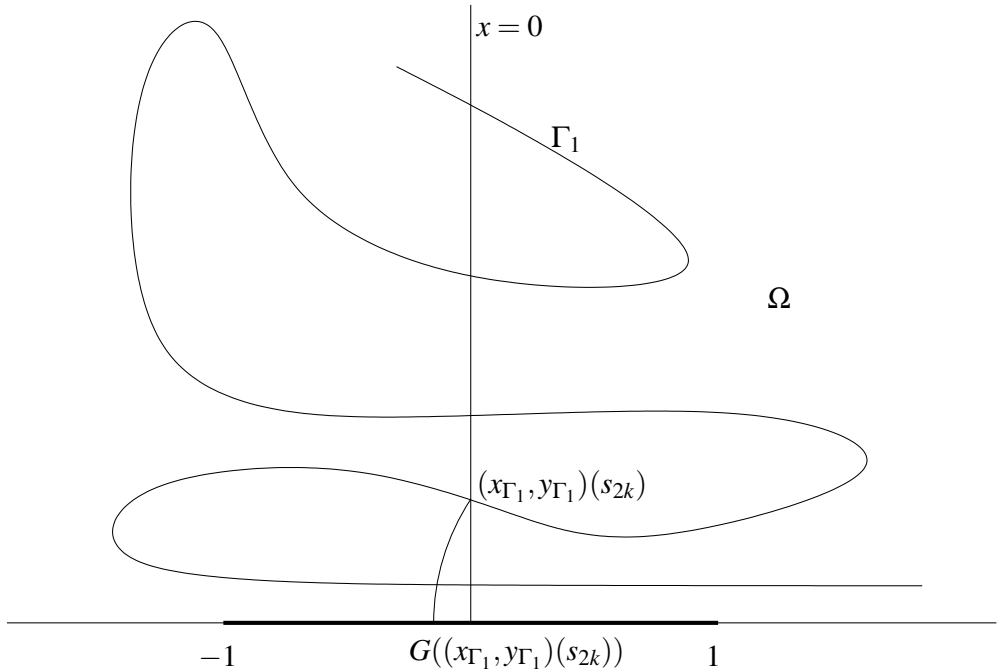}
\end{center}
\caption{If $\Gamma$ accumulates on $\partial_\infty\h^2$}\label{fig:fig3}
\end{figure}

\end{proof}

Using this asymptotic behavior and assuming~\eqref{hyp:h2}, we can conclude:

\begin{theorem}[BCN-Conjecture in $\mathbb H ^2$]\label{Th:dim2}
Let $\Omega \subset \mathbb H ^2$ be a connected domain with $C^2$ boundary and such that $\h^2
\setminus \overline \Omega $ is connected. If there exists a function
$u\in{C}^{2}(\Omega)$ that solves the OEP \eqref{1.4} with $f$ satisfying~\eqref{hyp:h2},
then $\partial\Ome$ has constant curvature.

More precisely, $\Ome$ must be either
\begin{itemize}
\item a geodesic disk or the complement of a geodesic disk or,
\item a horodisk or the complement of a horodisk or,
\item a half-space determined by a complete equidistant curve, i.e. a complete curve of
constant geodesic curvature $k_g\in[0,1)$.
\end{itemize}
Moreover, in each case, $u$ is invariant by the isometries fixing $\Ome$.
\end{theorem}

\begin{proof}
If $\partial\Ome$ is compact, then either $\Ome$ is compact and is a disk
\cite[Theorem~3.3]{EM} or $\Ome$ is unbounded and is the exterior of a disk by
Theorem~\ref{CorBounded}. In both cases, $u$ is radially symmetric.

If $\partial\Ome$ is not bounded, Lemma~\ref{lem:2point} implies that
$\partial_\infty(\partial\Ome)=\{a\}$ or $\{a,b\}\subset\partial_\infty\h^2$. In the first
case, either $\partial_\infty \Ome=\{a\}$ and $\Ome$ is a horodisk \cite[Theorem~3.8]{EM}
or $\partial_\infty(\h^2\setminus\overline\Ome)=\{a\}$ and $\Ome$ is the complement of a
horodisk by Theorem~\ref{OnePoint}. In both cases, $u$ is invariant by parabolic
isometries that fix $a$. If $\partial_\infty(\partial\Ome)=\{a,b\}$,
Theorem~\ref{Th:Graphical} implies that $\Ome$ is a a half-space determined by a complete
equidistant curve and $u$ is invariant by hyperbolic translations along the complete
geodesic joining $a ,b \in \partial _{\infty}\h^{2}$.
\end{proof}

\section*{Acknowledgments}
\renewcommand{\thesection}{\arabic{section}}
\renewcommand{\theequation}{\thesection.\arabic{equation}}
\setcounter{equation}{0} \setcounter{maintheorem}{0}

The first author, Jos\'{e} M. Espinar, is partially supported by Spanish MEC-FEDER Grant
MTM2013-43970-P; CNPq-Brazil Grants 405732/2013-9 and 14/2012 - Universal, Grant
302669/2011-6 - Produtividade; FAPERJ Grant 25/2014 - Jovem Cientista de Nosso Estado.
Alberto Farina is partially supported by the ERC grant EPSILON ({\it Elliptic Pde's and
Symmetry of Interfaces and Layers for Odd Nonlinearities}) and by the ERC grant COMPAT
({\it Complex Patterns for Strongly Interacting Dynamical Systems}). Laurent Mazet is
partially supported by the ANR-11-IS01-0002 grant.

 \end{document}